\documentclass{amsart}
\usepackage{amsmath,amsfonts}
\usepackage{mathrsfs}
\usepackage{upgreek}
\usepackage{graphicx}
\usepackage{url}
\usepackage{hyperref}
\usepackage{comment}
\excludecomment{commentcode}

\newcommand{\isom}{\cong}
\newcommand{\FF}{\mathbb{F}}
\newcommand{\PP}{\mathbb{P}}

\newcommand{\ZZ}{\mathbb{Z}}

\newcommand{\cK}{\mathscr{K}}
\newcommand{\fs}{\mathfrak{s}}
\newcommand{\bbmu}{\boldsymbol{\mu}}
\newcommand{\Mul}{\mathbf{M}}
\newcommand{\mul}{\mathbf{m}}
\newcommand{\Sqr}{\mathbf{S}}
\newcommand{\wx}{\omega}
\newcommand{\wy}{\overline{\wx}}

\newcommand{\ignore}[1]{}
\usepackage{pifont}
\newcommand{\cmark}{\ding{51}}%
\newcommand{\xmark}{\ding{55}}%

\newtheorem{theorem}{Theorem}
\newtheorem{lemma}[theorem]{Lemma}

\newtheorem{corollary}[theorem]{Corollary}
\newtheorem{definition}[theorem]{Definition}

\title{Twisted $\bbmu_4$-normal form for elliptic curves}
\author{David Kohel\\
Aix Marseille Univ, CNRS, Centrale Marseille, I2M, Marseille, France}

\begin{document}

\maketitle

\begin{abstract}
We introduce the twisted $\bbmu_4$-normal form for elliptic curves,
deriving in particular addition algorithms with complexity $9\Mul + 2\Sqr$
and doubling algorithms with complexity $2\Mul + 5\Sqr + 2\mul$
over a binary field.
Every ordinary elliptic curve over a finite field of characteristic~2
is isomorphic to one in this family.
This improvement to the addition algorithm, applicable to a larger class
of curves, is comparable to the $7\Mul + 2\Sqr$ achieved for the
$\bbmu_4$-normal form, and replaces the previously best known complexity
of $13\Mul + 3\Sqr$ on L\'opez-Dahab models applicable to these twisted curves.
The derived doubling algorithm is essentially optimal, without any
assumption of special cases.  We show moreover that the Montgomery
scalar multiplication with point recovery carries over to the twisted
models, giving symmetric scalar multiplication adapted to protect
against side channel attacks, with a cost of $4\Mul + 4\Sqr + 1\mul_t + 2\mul_c$ per bit.
In characteristic different from~2, we establish a linear isomorphism
with the twisted Edwards model over the base field.
This work, in complement to the introduction of $\bbmu_4$-normal form,
fills the lacuna in the body of work on efficient arithmetic on elliptic curves
over binary fields, explained by this common framework for elliptic
curves in $\bbmu_4$-normal form over a field of any characteristic.
The improvements are analogous to those which the Edwards
and twisted Edwards models achieved for elliptic curves over finite fields
of odd characteristic, and extend $\bbmu_4$-normal form to cover
the binary NIST curves.
\end{abstract}

\begin{commentcode}
\begin{sagesilent}
magma.attach_echidna()
magma.eval("ZZ := IntegerRing();")
\end{sagesilent}
\end{commentcode}

\section{Introduction}

Let $E$ be an elliptic curve with given embedding in $\PP^r$ and identity $O$.
The addition morphism $\mu: E \times E \rightarrow E$ is uniquely defined by the
pair $(E,O)$ but the homogeneous polynomial maps which determine $\mu$ are
not unique.  Let $x = (X_0,\dots,X_r)$ and $y = (Y_0,\dots,Y_r)$ be the coordinate
functions on the first and second factors, respectively.
We recall that an {\it addition law} (cf.~\cite{LangeRuppert}) is a bihomogenous
polynomial map $\fs = (p_0(x,y),\dots,p_r(x,y))$ which determines $\mu$ outside
of the common zero locus $p_0(x,y) = \cdots = p_r(x,y) = 0$. Such polynomial
addition laws play an important role in cryptography since they provide a means
of carrying out addition on $E$ without inversion in the base field.

In this work we generalize the algorithmic analysis of the $\bbmu_4$-normal form
to include twists.  The principal improvements are for binary curves, but we
are able to establish these results for a family which has good reduction and
efficient arithmetic over any field $k$, and in fact any ring.  We adopt the
notation $\Mul$ and $\Sqr$ for the complexity of multiplication and squaring
in $k$, and $\mul$ for multiplication by a fixed constant that depends
(polynomially) only on curve constants.

In Section~\ref{sec:mu4-normal-form} we introduce a hierarchy of curves
in $\bbmu_4$-normal form, according to the additional $4$-level structure
parametrized.
In referring to these families of curves, we give special attention to the
so-called split and semisplit variants, while using the generic term
$\bbmu_4$-normal form to refer to any of the families.
In particular their isomorphisms and addition laws are developed.
In the specialization to finite fields of characteristic~2, by extracting
square roots, we note that any of the families can be put in split
$\bbmu_4$-normal form, and the distinction is only one of symmetries and
optimization of the arithmetic.
In Section~\ref{sec:twisted-normal-forms}, we generalize this hierarchy
to quadratic twists, which, in order to hold in characteristic~2 are
defined in terms of Artin--Schreier extensions.
The next two sections deal with algorithms for these families of curves
over binary fields, particularly, their addition laws
in Section~\ref{sec:addition-algorithms} and their doubling algorithms
in Section~\ref{sec:binary-doubling}. These establish the main complexity
results of this work --- an improvement of the best known addition algorithms
on NIST curves to $9\Mul + 2\Sqr$ coupled with a doubling algorithm
of $2\Mul + 5\Sqr + 2\mul$.  These improvements are summarized in the
following table of complexities (see Section~\ref{sec:conclusion} for details).
\begin{center}
\begin{tabular}{|@{\;}c|@{\;}l|@{\;}l|@{\;}c|@{\;}c|}
\hline
Curve model    & Doubling & Addition & $\%$ & NIST\\
\hline
Lambda coordinates     & $3\Mul + 4\Sqr + 1\mul$     & $11\Mul + 2\Sqr$ & 100\% & \cmark\\
Binary Edwards ($d_1=d_2$) & $2\Mul + 5\Sqr + 2\mul$ & $16\Mul + 1\Sqr + 4\mul$ & 50\% & \xmark\\
L\'opez-Dahab ($a_2=0$)  & $2\Mul + 5\Sqr + 1\mul$ & $14\Mul + 3\Sqr$  & 50\% & \xmark \\
L\'opez-Dahab ($a_2=1$)  & $2\Mul + 4\Sqr + 2\mul$ & $13\Mul + 3\Sqr$  & 50\% & \cmark \\
Twisted $\bbmu_4$-normal form  & $2\Mul + 5\Sqr + 2\mul$ &\;\;$9\Mul + 2\Sqr$ & 100\% & \cmark \\
$\bbmu_4$-normal form  & $2\Mul + 5\Sqr + 2\mul$ &\;\;$7\Mul + 2\Sqr$  & 50\% & \xmark \\
\hline
\end{tabular}
\end{center}

\noindent
To complete the picture, we prove in Section~\ref{sec:montgomery-endomorphism}
that the Montgomery endomorphism and resulting complexity, as described
in Kohel~\cite{Kohel-Indocrypt} carry over to the twisted families,
which allows for an elementary and relatively efficient symmetric algorithm
for scalar multiplication which is well-adapted to protecting against
side-channel attacks.
While the most efficient arithmetic is achieved for curves for which
the curve coefficients are constructed such that the constant multiplications
are negligible, these extensions to twists provide efficient algorithms
for backward compatibility with binary NIST curves.

\section{The $\bbmu_4$-normal form}
\label{sec:mu4-normal-form}

In this section we recall the definition and construction of the family
of elliptic curves in (split) $\bbmu_4$-normal form.  The notion of a
canonical model of level $n$ was introduced in Kohel~\cite{Kohel} as an
elliptic curve $C/k$ in $\PP^{n-1}$ with subgroup scheme $G \isom \bbmu_n$
(a $k$-rational subgroup of the $n$-torsion subgroup $C[n]$ whose points
split in $k[\zeta_n]$, where $\zeta_n$ is an $n$-th root of unity in
$\bar{k}$) such that for $P = (x_0:x_1:\dots:x_{n-1})$ a generator $S$
of $G$ acts by $P + S = (x_0:\zeta_n^1 x_1:\dots:\zeta_n^{n-1} x_{n-1})$.
If, in addition, there exists a rational $n$-torsion point $T$ such that
$C[n] = \langle{S,T}\rangle,$ we say that the model is {\it split}
and impose the condition that $T$ acts by
a cyclic coordinate permutation.  Construction of the special cases
$n = 4$ and $n = 5$ were treated as examples in Kohel~\cite{Kohel}, and
the present work is concerned with a more in depth study of the former.

The Edwards curve $x^2 + y^2 = 1 + d x^2 y^2$ embeds in $\PP^3$ (by
$(1:x:y:xy)$ as the elliptic curve
$$
X_1^2 + X_2^2 = X_0^2 + d X_3^2,\ X_0 X_3 = X_1 X_2,
$$
with identity $O = (1:0:1:0)$. Such a model was studied by Hisil et
al.~\cite{Hisil-EdwardsRevisited}, as extended Edwards coordinates,
and admits the fastest known arithmetic on such curves.
The twist by $a$, in extended coordinates, is the twisted Edwards
curve
$$
a X_1^2 + X_2^2 = X_0^2 + ad X_3^2,\ X_0 X_3 = X_1 X_2
$$
with parameters $(a,ad)$. For the special case $(a,ad) = (-1,-16r)$,
the change of variables
$$
(X_0:X_1:X_2:X_3) \mapsto (X_0,X_1+X_2,4X_3,-X_1+X_2).
$$
has image the canonical model of level~$4$ above. The normalization to
have good reduction at~$2$ (by setting $d = 16r$ and the coefficient
of $X_3$) as well as the following refined hierarchy of curves appears
in Kohel~\cite{Kohel-AGCT}, and the subsequent article~\cite{Kohel-Indocrypt}
treated only the properties of this hierarchy over fields of
characteristic~2.

\begin{definition}
An elliptic curve in {\em $\bbmu_4$-normal form} is a genus one curve in
the family
$$
X_0^2 - r X_2^2 = X_1 X_3,\ X_1^2 - X_3^2 = X_0 X_2
$$
with base point $O = (1:1:0:1)$.
An elliptic curve in {\em semisplit $\bbmu_4$-normal form}
is a genus one curve in the family
$$
X_0^2 - X_2^2 = X_1 X_3,\ X_1^2 - X_3^2 = s X_0 X_2,
$$
with identity $O = (1:1:0:1)$, and an elliptic curve is in
{\em split $\bbmu_4$-normal form} if it takes the form
$$
X_0^2 - X_2^2 = c^2 X_1 X_3,\ X_1^2 - X_3^2 = c^2 X_0 X_2.
$$
with identity $O = (c:1:0:1)$.
\end{definition}

Setting $s = c^4$, the transformation
$$
(X_0:X_1:X_2:X_3) \mapsto (X_0:cX_1:cX_2:X_3)
$$
maps the split $\bbmu_4$-normal form to semisplit $\bbmu_4$-normal form
with parameter $s$, and setting $r = 1/s^2$, the transformation
$$
(X_0:X_1:X_2:X_3) \mapsto (X_0:X_1:sX_2:X_3)
$$
maps the semisplit $\bbmu_4$-normal form to $\bbmu_4$-normal form with
parameter $r$.
The names for the $\bbmu_4$-normal forms of a curve $C/k$ in $\PP^3$,
recognize the existence of $\bbmu_4$ as a $k$-rational subgroup
scheme of $C[4]$, and secondly, its role as defining the embedding
class of $C$ in $\PP^3$, namely it is cut out by the hyperplane
$X_2 = 0$ in~$\PP^3$.

\begin{lemma}
Let $C$ be a curve in $\bbmu_4$-normal form, semi-split $\bbmu_4$-normal
form, or split $\bbmu_4$-normal form, with identity $(e,1,0,1)$.
For any extension containing a square root $i$ of $-1$, the point
$S = (e:i:0:-i)$ is a point of order $4$ acting by the coordinate scaling
$(x_0:x_1:x_2:x_3) \mapsto (x_0:i x_1:-x_2:-ix_3)$.  In particular,
$$
\{ (e:1:0:1), (e:i:0:-i), (e:-1:0:-1), (e:i:0:-i) \},
$$
is a subgroup of $C[4] \subseteq C(\bar{k})$.
\end{lemma}

The semisplit $\bbmu_4$-normal form with square parameter $s = t^2$ admits
a $4$-torsion point $(1:t:1:0)$ acting by scaled coordinate permutation.
After a further quadratic extension $t = c^2$, the split $\bbmu_4$-normal
form admits the constant group scheme $\ZZ/4\ZZ$ acting by signed coordinate
permutation.

\begin{lemma}
Let $C/k$ be an elliptic curve in split $\bbmu_4$-normal form with identity
$O = (c:1:0:1)$.  Then $T = (1:c:1:0)$ is a point in $C[4]$, and translation
by $T$ induces the signed coordinate permutation
$$
(x_0:x_1:x_2:x_3) \longmapsto (x_3:x_0:x_1:-x_2)
$$
on $C$.
\end{lemma}

This gives the structure of a group $C[4] \isom \bbmu_4 \times \ZZ/4\ZZ$,
whose generators $S$ and $T$ are induced by the matrix actions
$$
A(S) =
\left(\begin{array}{@{\;}rrr@{\;}r@{\;}}
 1 & 0 & 0 & 0\\
 0 & i & 0 & 0\\
 0 & 0 & 1 & 0\\
 0 & 0 & 0 &-i
\end{array}\right)
\mbox{ and }
A(T) =
\left(\begin{array}{@{\;}rrr@{\;}r@{\;}}
 0 & 1 & 0 & 0\\
 0 & 0 & 1 & 0\\
 0 & 0 & 0 &-1\\
 1 & 0 & 0 & 0
\end{array}\right)
$$
on $C$ such that $A(S)A(T) = iA(T)A(S)$.
We can now state the structure of addition laws for the split $\bbmu_4$-normal
form and its relation to the torsion action described above.

\begin{commentcode}
\begin{table}
\caption{Torsion structure of split $\bbmu_4$-normal form.}
\label{code:torsion-split-mu_4-normal-form}
\begin{sageblock}
magma.eval("""
KK<i> := CyclotomicField(4);
FF<c> := FunctionField(KK);
PP<X0,X1,X2,X3> := ProjectiveSpace(FF,3);
EE := EllipticCurve_Split_Mu4_NormalForm(PP,c);
OE := EE!0;
SE := EE![c,i,0,-i];
TE := EE![1,c,1,0];
tau := map< EE->EE | [X3,X0,X1,-X2] >;
assert tau(OE) eq TE;
sig := map< EE->EE | [X0,i*X1,-X2,-i*X3] >;
assert sig(OE) eq SE;
""")
\end{sageblock}
\end{table}
\end{commentcode}

\begin{theorem}
\label{thm:mu4-addition-laws}
Let $C$ be an elliptic curve in split $\bbmu_4$-normal form:
$$
X_0^2-X_2^2 = c^2\,X_1 X_3,\
X_1^2-X_3^2 = c^2\,X_0 X_2,\
O = (c:1:0:1),
$$
and set $U_{jk} = X_j Y_k$. A complete basis of addition laws
of bidegree $(2,2)$ is given by:
$$
\begin{array}{l}
\fs_0 = (
  U_{13}^2 - U_{31}^2,\
  c ( U_{13} U_{20} - U_{31} U_{02}),\
  U_{20}^2 - U_{02}^2,\
  c (U_{20} U_{31} - U_{13} U_{02})
),\\
\fs_1 = (
  c (U_{03} U_{10} + U_{21} U_{32}),\
  U_{10}^2 - U_{32}^2,\
  c (U_{03} U_{32} + U_{10} U_{21}),\
  U_{03}^2 - U_{21}^2
),\\
\fs_2 =
(
  U_{00}^2 - U_{22}^2,\
  c (U_{00} U_{11} - U_{22} U_{33}),\
  U_{11}^2 - U_{33}^2,\
  c (U_{00} U_{33} - U_{11} U_{22})
),\\
\fs_3 = (
  c (U_{01} U_{30} + U_{12} U_{23}),\
  U_{01}^2 - U_{23}^2,\
  c (U_{01} U_{12} + U_{23} U_{30}),\
  U_{30}^2 - U_{12}^2
).
\end{array}
$$
The exceptional divisor of the addition law $\fs_\ell$ is
$\sum_{k=0}^3 \Delta_{kS+{\ell}T}$, where $S$ and $T$ are
the $4$-torsion points $(c:i:0:-i)$ and $(1:c:1:0)$,
and the divisors $\sum_{k=0}^3 (kS+{\ell}T)$ are determined
by $X_{\ell+2} = 0$.
In particular, any pair of the above addition laws
provides a complete system of addition laws.
\end{theorem}

\begin{proof}
This appears as Theorem~44 of Kohel~\cite{Kohel} for the
$\bbmu_4$-normal form, subject to the scalar renormalizations
indicated above.
The exceptional divisor is a sum of four curves of the form
$\Delta_P$ by Theorem~10 of Kohel~\cite{Kohel}, and the
points $P$ can be determined by intersection with
$H = C \times \{O\}$ using Corollary~11 of Kohel~\cite{Kohel}.
Taking the particular case $\fs_2$, we substitute
$(Y_0,Y_1,Y_2,Y_3) = (c,1,0,1)$ to obtain
$(U_{00},U_{11},U_{22},U_{33}) = (cX_0,X_1,0,X_3)$, and hence
$$
(U_{00}^2 - U_{22}^2,\
U_{00} U_{11} - U_{22} U_{33},\
U_{11}^2 - U_{33}^2,\
U_{00} U_{33} - U_{22} U_{11}),
$$
which equals
$$
(c^2 X_0^2, c X_0 X_1, X_1^2 - X_3^2, c X_0 X_3)
= (c^2 X_0^2, c X_0 X_1, c^2 X_0 X_2, c X_0 X_3).
$$
These coordinate functions cuts out the divisor $X_0 = 0$ with
support on the points $kS + 2T$, $0 \le k < 4$, where $2T = (0:-1:-c:1)$.
The final statement follows since the exceptional divisors
are disjoint.
\qed
\end{proof}

The above basis of addition laws can be generated by any one of
the four, by means of signed coordinate permutation on input
and output determined by the action of the $4$-torsion group.
Denote translation by $S$ and $T$ by $\sigma$ and $\uptau$,
respectively, given by the coordinate scalings and permutations
$$
\begin{array}{r@{\;}c@{\;}l}
\sigma(X_0:X_1:X_2:X_3) & = & (X_0:iX_1:-X_2:-iX_3),\\
\uptau(X_0:X_1:X_2:X_3) & = & (X_3:X_0:X_1:-X_2),
\end{array}
$$
as noted above.
The set $\{ \fs_0,\fs_1,\fs_2,\fs_3 \}$ forms a basis of eigenvectors
for the action of $\sigma$. More precisely for all $(j,k,\ell)$, we have
$$
\fs_\ell = (-1)^{j+k+\ell} \sigma^{-j-k} \circ \fs_\ell \circ (\sigma^j \times \sigma^k).
$$
Then $\uptau$, which projectively commutes with $\sigma$, acts by
a scaled coordinate permutation
$$
\fs_{\ell-j-k} = \uptau^{-j-k} \circ \fs_\ell \circ (\uptau^j \times \uptau^k),
$$
consistent with the action on the exceptional divisors (see Lemma~31
of Kohel~\cite{Kohel}).

Consequently, the complexity of evaluation of any of these addition laws
is computationally equivalent, since they differ only by a signed coordinate
permutation on input and output.

\begin{commentcode}
\begin{table}
\caption{Addition law structure of split $\bbmu_4$-normal form.}
\label{code:torsion-split-addition-law-action}
\begin{sageblock}
magma.eval("""
KK<i> := CyclotomicField(4);
FF<c> := FunctionField(KK);
EE := EllipticCurve_Split_Mu4_NormalForm(c);
mu := AdditionMorphism(EE);
EExEE := Domain(mu);
PPxPP<X0,X1,X2,X3,Y0,Y1,Y2,Y3> := AmbientSpace(EExEE);
XX := [X0,X1,X2,X3]; YY := [Y0,Y1,Y2,Y3];
s := func< i | ((i-1) mod 4) + 1 >;
U00, U11, U22, U33 := Explode([ XX[s(i+0)]*YY[s(i+0)] : i in [1..4] ]);
U01, U12, U23, U30 := Explode([ XX[s(i+0)]*YY[s(i+1)] : i in [1..4] ]);
U02, U13, U20, U31 := Explode([ XX[s(i+0)]*YY[s(i+2)] : i in [1..4] ]);
U03, U10, U21, U32 := Explode([ XX[s(i+0)]*YY[s(i+3)] : i in [1..4] ]);
BB := [
    [ U13^2 - U31^2, c*( U13*U20 - U31*U02), U20^2 - U02^2,-c*(U13*U02 - U20*U31) ],
    [ c*(U03*U10 + U21*U32), U10^2 - U32^2, c*(U03*U32 + U10*U21), U03^2 - U21^2 ],
    [ U00^2 - U22^2, c*(U00*U11 - U22*U33), U11^2 - U33^2, c*(U00*U33 - U11*U22) ],
    [ c*(U01*U30 + U12*U23), U01^2 - U23^2, c*(U01*U12 + U23*U30), U30^2 - U12^2 ]
    ];
mu := map< EExEE->EE | BB >;
function tau(S,i)
    // Operates on the indices...
    if (i mod 4) eq 0 then return S; end if;
    S_tau := S[[4,1,2,3]]; S_tau[4] *:= -1;
    return tau(S_tau,i-1);
end function;
function tau_tau(S,i,j)
    // Operates on the coordinate functions
    PP := Universe(S);
    XX := [PP.k : k in [1..4]]; XX_i := tau(XX,i);
    YY := [PP.k : k in [5..8]]; YY_j := tau(YY,j);
    return [ f(XX_i cat YY_j) : f in S ];
end function;
B1 := BB[1];
A0 := Matrix([[Index(BB,[+f : f in tau(tau_tau(B1,i,j),-i-j)]) : i in [0..3]] : j in [0..3]]);
A1 := Matrix([[Index(BB,[-f : f in tau(tau_tau(B1,i,j),-i-j)]) : i in [0..3]] : j in [0..3]]);
A0 - A1;
""")
\end{sageblock}
\end{table}
\end{commentcode}

\begin{corollary}
Let $C$ be an elliptic curve in split $\bbmu_4$-normal form.
There exist algorithms for addition with complexity
$9\Mul + 2\mul$ over any ring,
$8\Mul + 2\mul$ over a ring in which 2 is a unit, and
$7\Mul + 2\Sqr + 2\mul$ over a ring of characteristic $2$.
\end{corollary}

\begin{proof}
We determine the complexity of an algorithm for the evaluation
of the addition law $\fs_2$:
$$
(Z_0,Z_1,Z_2,Z_3) = (
  U_{00}^2 - U_{22}^2,\,
  c (U_{00} U_{11} - U_{22} U_{33}),\,
  U_{11}^2 - U_{33}^2,\,
  c (U_{00} U_{33} - U_{11} U_{22})\,
),
$$
recalling that each of the given addition laws in the basis has
equivalent evaluation. Over a general ring, we make use of the
equalities:
$$
\begin{array}{l}
Z_0 = U_{00}^2 - U_{22}^2 = (U_{00} - U_{22})(U_{00} + U_{22}),\\
Z_2 = U_{11}^2 - U_{33}^2 = (U_{11} - U_{33})(U_{11} + U_{33}),
\end{array}
$$
and
$$
\begin{array}{l}
Z_1 + Z_3 = c(U_{00} U_{11} - U_{22} U_{33}) + c(U_{00} U_{33} - U_{22} U_{11})
          = c(U_{00} - U_{22})(U_{11} + U_{33}),\\
Z_1 - Z_3 = c(U_{00} U_{11} - U_{22} U_{33}) - c(U_{00} U_{33} - U_{22} U_{11})
          = c(U_{00} + U_{22})(U_{11} - U_{33}),
\end{array}
$$
using $1\Mul + 1\mul$ each for their evaluation.
\begin{itemize}
\item Evaluate $U_{jj} = X_j Y_j$, for $1 \le j \le 4$, with $4\Mul$.
\item Evaluate $(Z_0,\,Z_2) = (U_{00}^2 - U_{22}^2,\, U_{11}^2 - U_{33}^2)$ with $2\Mul$.
\item Evaluate $A = c (U_{00} - U_{22})(U_{11} + U_{33})$ using $1\Mul + 1\mul$.
\item Compute $Z_1 = c (U_{00} U_{11} - U_{22} U_{33})$ and set $Z_3 = A - Z_1$ with $2\Mul + 1\mul$.
\end{itemize}
This yields the desired complexity $9\Mul + 2\mul$ over any ring.
If $2$ is a unit (and assuming a neglible cost of multiplying by $2$),
we replace the last line with two steps:
\begin{itemize}
\item Evaluate $B = c (U_{00} + U_{22})(U_{11} - U_{33})$ using $1\Mul + 1\mul$.
\item Compute $(2Z_1,2Z_3) = (A + B, A - B)$ and scale $(Z_0,Z_2)$ by $2$,
\end{itemize}
which gives a complexity of $8\Mul + 2\mul$.  This yields an algorithm
essentially equivalent to that Hisil et al.~\cite{Hisil-EdwardsRevisited}
under the linear isomorphism with the $-1$-twist of Edwards normal form.
Finally if the characteristic is $2$, the result $7\Mul + 2\Sqr + 2\mul$
of Kohel~\cite{Kohel-Indocrypt} is obtained by replacing $2\Mul$ by $2\Sqr$
for the evaluation of $(Z_0,\, Z_2)$ in the generic algorithm.
\qed
\end{proof}


Before considering the twisted forms, we determine the base complexity
of doubling for the split $\bbmu_4$-normal form.

\begin{corollary}
Let $C$ be an elliptic curve in split $\bbmu_4$-normal form.
There exist algorithms for doubling with complexity
$5\Mul + 4\Sqr + 2\mul$ over any ring,
$4\Mul + 4\Sqr + 2\mul$ over a ring in which 2 is a unit, and
$2\Mul + 5\Sqr + 7\mul$ over a ring of characteristic~$2$.
\end{corollary}

\begin{proof}
The specialization of the addition law $\fs_2$ to the diagonal gives
the forms for doubling
$$
(
  X_0^4 - X_2^4,\,
  c (X_0^2 X_1^2 - X_2^2 X_3^2),\,
  X_1^4 - X_3^4,\,
  c (X_0^2 X_3^2 - X_1^2 X_2^2)\,
).
$$
which we can evaluate as follows:
\begin{itemize}
\item Evaluate $X_j^2$, for $1 \le j \le 4$, with $4\Sqr$.
\item Evaluate $(Z_0,\,Z_2) = (X_0^4 - X_2^4,\, X_1^4 - X_3^4)$ with $2\Mul$.
\item Evaluate $A = c (X_0^2 - X_2^2)(X_1^2 + X_3^2)$ using $1\Mul + 1\mul$.
\item Compute $Z_1 = c (X_0^2 X_1^2 - X_2^2 X_3^2)$ and set $Z_3 = A - Z_1$ with $2\Mul + 1\mul$.
\end{itemize}
This gives the result of $5\Mul + 4\Sqr + 2\mul$ over any ring.
As above, when $2$ is a unit, we replace the last line with the two steps:
\begin{itemize}
\item Evaluate $B = c (X_0^2 + X_2^2)(X_1^2 - X_3^2)$ using $1\Mul + 1\mul$.
\item Compute $(2Z_1,2Z_3) = (A + B, A - B)$ and scale $(Z_0,Z_2)$ by $2$.
\end{itemize}
This reduces the complexity by $1\Mul$.
In characteristic~$2$, the general algorithm specializes to $3\Mul + 6\Sqr + 2\mul$,
but Kohel~\cite{Kohel-Indocrypt} provides an algorithm with better complexity of
$2\Mul + 5\Sqr + 7\mul$ (reduced by $5\mul$ on the semisplit model).
\qed
\end{proof}

In the next section, we introduce the twists of these $\bbmu_4$-normal forms,
and derive efficient algorithms for their arithmetic.

\section{Twisted normal forms}
\label{sec:twisted-normal-forms}

A quadratic twist of an elliptic curve is determined by a non-rational isomorphism
defined over a quadratic extension $k[\alpha]/k$.  In odd characteristic one can
take an 
extension defined by $\alpha^2 = a$, but in characteristic~$2$, the
general form of a quadratic extension is 
$k[\wx]/k$ where $\wx^2 - \wx = a$ for some $a$ in $k$.
The normal forms defined above both impose the existence of a $k$-rational point
of order $4$. 

Over a finite field of characteristic~2, the existence of a $4$-torsion point is
a weaker constraint than for odd characteristic, since if $E/k$ is an ordinary
elliptic curve over a finite field of characteristic~$2$, there necessarily exists
a $2$-torsion point.
Moreover, if $E$ does not admit a $k$-rational $4$-torsion point
and $|k| > 2$, then its quadratic twist does.

We recall that for an elliptic curve in Weierstrass form,
$$
E : Y^2 Z + (a_1 X + a_3 Z) Y Z = X^3 + a_2 X^2 Z + a_4 X Z^2 + a_6 Z^3,
$$
the quadratic twist by $k[\wx]/k$ is given by
$$
E^t : Y^2 Z + (a_1 X + a_3 Z) Y Z = X^3 + a_2 X^2 Z + a_4 X Z^2 + a_6 Z^3 + a (a_1 X + a_3 Z)^2 Z,
$$
with isomorphism $\tau(X:Y:Z) = (X:-Y-\wx(a_1X+a_3Z):Z)$, which 
satisfies $\tau^{\sigma} = -\tau$,
where $\sigma$ is the nontrivial automorphism of $k[\wx]/k$. The objective
here is to describe the quadratic twists in the case of the normal forms
defined above.

With a view towards cryptography, the binary NIST curves are of the form
$y^2 + xy = x^3 + ax^2 + b$, with $a = 1$ and group order $2n$, whose
quadratic twist is the curve with $a = 0$ which admits a point of order $4$.
While the latter admit an isomorphism a curve in $\bbmu_4$-normal form,
to describe the others, we must represent them as quadratic twists.


\subsection*{The twisted $\bbmu_4$-normal form}

In what follows we let $k[\wx]/k$ be the quadratic extension given by
$\wx^2 - \wx = a$, and set $\wy = 1 - \wx$ and $\delta = \wx - \wy$.
In order to have the widest possible applicability, we describe the
quadratic twists with respect to any ring or field~$k$.
The discriminant of the extension is $D = \delta^2 = 1+4a$.  When
$2$ is invertible we can speak of a twist by $D$, but in general
we refer to $a$ as the twisting parameter.
While admitting general rings, all formulas hold over a field of
characteristic $2$, and we investigate optimizations in this case.

\begin{theorem}
\label{thm:twisted-split-mu4-curve}
Let $C/k$ be an elliptic curve in $\bbmu_4$-normal form,
semisplit $\bbmu_4$-normal form, or split $\bbmu_4$-normal form,
given respectively by
$$
\begin{array}{cl}
X_0^2 - r\,X_2^2 = X_1 X_3,\
X_1^2 - X_3^2 = X_0 X_2,\
& O = (1:1:0:1),\\
X_0^2 - X_2^2 = X_1 X_3,\
X_1^2 - X_3^2 = s\,X_0 X_2,\
& O = (1:1:0:1),\\
X_0^2 - X_2^2 = c^2\,X_1 X_3,\
X_1^2 - X_3^2 = c^2\,X_0 X_2,\
& O = (c:1:0:1).
\end{array}
$$
The quadratic twist $C^t$ of $C$ by $k[\wx]$, where $\wx^2 - \wx = a$, is
given by
$$
\begin{array}{cl}
X_0^2 - D r\,X_2^2 = X_1 X_3 - a (X_1 - X_3)^2,\
X_1^2 - X_3^2 = X_0 X_2,\\
X_0^2 - D X_2^2 = X_1 X_3  - a (X_1 - X_3)^2,\
X_1^2 - X_3^2 = s\,X_0 X_2,\\
X_0^2 - D X_2^2 = c^2 (X_1 X_3 - a (X_1 - X_3)^2),\
X_1^2 - X_3^2 = c^2 X_0 X_2,
\end{array}
$$
with identities $O = (1:1:0:1)$, $O = (1:1:0:1)$ and $O = (c:1:0:1)$,
respectively. In each case, the twisting isomorphism $\tau: C \rightarrow C^t$
is given by
$$
(X_0:X_1:X_2:X_3) \longmapsto
(\delta X_0 : \wx X_1 - \wy X_3 : X_2 : \wx X_3 - \wy X_1),
$$
with inverse sending $(X_0:X_1:X_2:X_3)$ to
$
(X_0 : \wx X_1 + \wy X_3 : \delta X_2 : \wy X_1 + \wy X_3).
$
\end{theorem}

\begin{proof}
Since the inverse morphism is $[-1](X_0:X_1:X_2:X_3) = (X_0:X_3:-X_2:X_1)$,
the twisting morphism satisfies $\tau^\sigma = [-1]\circ\tau$ where
$\sigma$ is the nontrivial automorphism of $k[\wx]/k$.  Consequently,
the image $C^t$ is a twist of $C$. The form of the inverse is obtained
by matrix inversion.
\qed
\end{proof}

\begin{commentcode}
\noindent{\bf Code verification.} The twisting isomorphisms are verified symbolically
in the code of Tables~\ref{code:twisting-isom-init}, \ref{code:twisting-isom-mu4},
\ref{code:twisting-isom-semisplit-mu4}, and~\ref{code:twisting-isom-split-mu4}.

\begin{table} 
\caption{Initialization of twisting extension and projective space.}
\label{code:twisting-isom-init}
\begin{sageblock}
# First we initialize the generic field, twisting extension,
# and create the ambient projective space.
magma.eval("""
FF<c,a> := FunctionField(ZZ,2);
r := 1/c^8; s := c^4; D := 4*a+1;
PF<x> := PolynomialRing(FF);
KK<w> := quo< PF | x^2-x-a >;
w1 := w; w2 := 1-w; dd := w1-w2; assert dd^2 eq D;
PP<X0,X1,X2,X3> := ProjectiveSpace(KK,3);
""")
\end{sageblock}
\end{table}

\begin{table} 
\caption{Verification of twisting isomorphism for $\bbmu_4$-normal form.}
\label{code:twisting-isom-mu4}
\begin{sageblock}
# Verify the theorem for the mu_4-normal form
magma.eval("""
CC := EllipticCurve_Mu4_NormalForm(PP,r);
OC := CC!0;
assert CC eq Curve(PP,[X0^2-r*X2^2-X1*X3,X1^2-X3^2-X0*X2]);
Ct := EllipticCurve_Twisted_Mu4_NormalForm(PP,r,a);
Ot := Ct!0;
assert Ct eq Curve(PP,[X0^2-D*r*X2^2-(X1*X3-a*(X1-X3)^2),X1^2-X3^2-X0*X2]);
tau_CC_Ct := map< CC->Ct |
    [dd*X0,w1*X1-w2*X3,X2,w1*X3-w2*X1],
    [X0,w1*X1+w2*X3,dd*X2,w2*X1+w1*X3] >;
tau_Ct_CC := Inverse(tau_CC_Ct);
assert tau_CC_Ct(OC) eq Ot;
assert tau_Ct_CC(Ot) eq OC;
""")
\end{sageblock}
\end{table}

\begin{table} 
\caption{Verification of twisting isomorphism for semisplit $\bbmu_4$-normal form.}
\label{code:twisting-isom-semisplit-mu4}
\begin{sageblock}
# Verify the theorem for the semisplit mu_4-normal form
magma.eval("""
CC := EllipticCurve_Semisplit_Mu4_NormalForm(PP,s);
OC := CC!0;
assert CC eq Curve(PP,[X0^2-X2^2-X1*X3,X1^2-X3^2-s*X0*X2]);
Ct := EllipticCurve_Twisted_Semisplit_Mu4_NormalForm(PP,s,a);
Ot := Ct!0;
assert Ct eq Curve(PP,[X0^2-D*X2^2-X1*X3+a*(X1-X3)^2,X1^2-X3^2-s*X0*X2]);
tau_CC_Ct := map< CC->Ct |
    [dd*X0,w1*X1-w2*X3,X2,w1*X3-w2*X1],
    [X0,w1*X1+w2*X3,dd*X2,w2*X1+w1*X3] >;
tau_Ct_CC := Inverse(tau_CC_Ct);
assert tau_CC_Ct(OC) eq Ot;
assert tau_Ct_CC(Ot) eq OC;
""")
\end{sageblock}
\end{table}

\begin{table} 
\caption{Verification of twisting isomorphism for split $\bbmu_4$-normal form.}
\label{code:twisting-isom-split-mu4}
\begin{sageblock}
# Verify the theorem for the split mu_4-normal form
magma.eval("""
CC := EllipticCurve_Split_Mu4_NormalForm(PP,c);
OC := CC!0;
assert CC eq Curve(PP,[X0^2-X2^2-c^2*X1*X3,X1^2-X3^2-c^2*X0*X2]);
Ct := EllipticCurve_Twisted_Split_Mu4_NormalForm(PP,c,a);
Ot := Ct!0;
assert Ct eq Curve(PP,[X0^2-D*X2^2-c^2*(X1*X3-a*(X1-X3)^2),X1^2-X3^2-c^2*X0*X2]);
tau_CC_Ct := map< CC->Ct |
    [dd*X0,w1*X1-w2*X3,X2,w1*X3-w2*X1],
    [X0,w1*X1+w2*X3,dd*X2,w2*X1+w1*X3] >;
tau_Ct_CC := Inverse(tau_CC_Ct);
assert tau_CC_Ct(OC) eq Ot;
assert tau_Ct_CC(Ot) eq OC;
""")
\end{sageblock}
\end{table}
\end{commentcode}

\noindent{\bf Remark.}
In characteristic $2$ we have $D = \delta = 1$, and the twisted split
$\bbmu_4$-normal form is
$
X_0^2 + X_2^2 = c^2 (X_1 X_3 + a (X_1 + X_3)^2),\ X_1^2 + X_3^2 = c^2 X_0 X_2,
$
with associated twisting morphism
$$
(X_0:X_1:X_2:X_3) \longmapsto (X_0 : \wy X_1 + \wx X_3 : X_2 : \wx X_1 + \wy X_3).
$$
Over a field of characteristic different from $2$, we have an isomorphism
with the twisted Edwards normal form.

\begin{theorem}
\label{thm:twisted-mu4-Edwards-isoms}
Let $C^t$ be an elliptic curve in twisted $\bbmu_4$-normal form
$$
X_0^2 - D r X_2^2 = X_1 X_3 - a (X_1 - X_3)^2,\ X_1^2 - X_3^2 = X_0 X_2,
$$
with parameters $(r,a)$ over a field of characteristic different from $2$.
Then $C^t$ is isomorphic to the twisted Edwards curve
$$
X_0^2 - 16 DrX_3^2 = - D X_1^2 + X_2^2
$$
with parameters $(-D,-16Dr)$, via the isomorphism $C^t \rightarrow E$:
$$
(X_0:X_1:X_2:X_3) \longmapsto (4 X_0 : 2(X_1 - X_3) : 2(X_1 + X_3) : X_2),
$$
and inverse
$$
(X_0:X_1:X_2:X_3) \longmapsto (X_0 : X_1 + X_2 : 4 X_3 : -X_1 + X_2).
$$
\end{theorem}

\begin{proof}
The linear transformation is the compositum of the above linear
transformations with the morphism
$
(X_0:X_1:X_2:X_3) \longmapsto ( \delta X_0 : X_1  : \delta X_2 : X_3)
$
from the Edwards curve to its twist.
\qed
\end{proof}

\begin{commentcode}
\begin{table} 
\caption{Isomorphism of $\bbmu_4$-normal form with Edwards normal form}
\label{code:isom-twisted-split-mu4-twisted-edwards}
\begin{sageblock}
magma.eval("""
Ct := EllipticCurve_Twisted_Mu4_NormalForm(PP,r,a);
Et := EllipticCurve_Twisted_Edwards_NormalForm(PP,-16*D*r,-D);
BCE := [4*X0, 2*(X1 - X3), 2*(X1 + X3), X2];
BEC := [X0, X1 + X2, 4*X3, -X1 + X2];
iso_Ct_Et := map< Ct->Et | BCE, BEC >;
""")
\end{sageblock}
\end{table}
\end{commentcode}

\noindent
For completeness we provide an isomorphic model in Weierstrass form:

\begin{theorem}
Let $C^t$ be an elliptic curve in twisted split $\bbmu_4$-normal form
with parameters $(r,a)$.  Then $C^t$ is isomorphic to the elliptic curve
$$
y^2 + xy = x^3 + (a-8Dr)x^2 + 2D^2r(8r-3)x - D^3r(1-4r)
$$
in Weierstrass form, where $D = 4a+1$.  The isomorphism is given by
the map which sends $(X_0:X_1:X_2:X_3)$ to
$$
\left(
D\big( U_0 - 4r (U_0 + U_2) \big) :
D\big( U_1 - 2r (8 U_1 + 2 U_0 - U_2) \big) :
U_2 - 2 U_0) \right),
$$
where $(U_0,U_1,U_2,U_3) = (X_1 - X_3, X_0 + X_3, X_2, X_1 + X_3)$.
\end{theorem}

\begin{proof}
A symbolic verification is carried out by the Echidna code~\cite{Kohel-Echidna}
implemented in Magma~\cite{Magma}.
\qed
\end{proof}

\begin{commentcode}
\begin{proof}
The symbolic verification is carried out by the code of Table~\ref{weierstrass-isom-charp}.
\qed
\end{proof}

\begin{table} 
\caption{Isomorphism with Weierstrass model}
\label{weierstrass-isom-charp}
\begin{sageblock}
magma.eval("""
FF<r,a> := FunctionField(ZZ,2); D := 4*a+1;
PP<X0,X1,X2,X3> := ProjectiveSpace(FF,3);
CT := EllipticCurve_Twisted_Mu4_NormalForm(PP,r,a);
WT := EllipticCurve([1,a-8*D*r,0,2*D^2*r*(8*r-3),-D^3*r*(1-4*r)]);
U0,U1,U2,U3 := Explode([X1-X3, X0+X3, X2, X1+X3]);
BB_CT_WT := [D*(U0-4*r*(2*U0+U2)),D*(U1-2*r*(8*U1+2*U0-U2)),U2-2*U0];
P2<X,Y,Z> := AmbientSpace(WT);
S0, S1 := Explode([X + 4*D*r*Z, 2*X + D*(1-8*r)*Z]);
T0, T1 := Explode([X + 2*Y, 4*X + D*Z]);
BB_WT_CT := [S0*T0, (S1*T0+S0*T1)/2, S1*T1, (S1*T0-S0*T1)/2];
iso_CT_WT := map< CT->WT | BB_CT_WT, BB_WT_CT >;
""")
\end{sageblock}
\end{table}
\end{commentcode}

\noindent
Specializing to characteristic~$2$, we obtain the following corollary.

\begin{corollary}
\label{cor:binary-weierstrass}
Let $C^t$ be a binary elliptic curve in twisted $\bbmu_4$-normal form
$$
X_0^2 + b X_2^2 = X_1 X_3 + a X_0 X_2,\ X_1^2 + X_3^2 = X_0 X_2,
$$
with parameters $(r,a) = (b,a)$. Then $C^t$ is isomorphic to the
elliptic curve
$$
y^2 + xy = x^3 + ax^2 + b,
$$
in Weierstrass form via the map
$(X_0:X_1:X_2:X_3) \mapsto ( X_1 + X_3 : X_0 + X_1 : X_2)$.
On affine points $(x,y)$ the inverse is
$
(x,y) \longmapsto (x^2 : x^2 + y : 1 : x^2 + y + x).
$
\end{corollary}

\begin{proof}
By the previous theorem, since $D = 1$ in characteristic~2, the
Weierstrass model simplifies to $y^2 + xy = x^3 + ax^2 + b$, and
the map to
$$
(X_0:X_1:X_2:X_3) \longmapsto
(U_0:U_1:U_2) = \left( X_1 + X_3 : X_0 + X_1 : X_2 \right).
$$
The given map on affine points is easily seen to be a birational
inverse, valid for $X_2 = 1$, in view of the relation $(X_1 + X_3)^2
= X_0 X_2$, well-defined outside the identity. Consequently,
it extends uniquely to an isomorphism.
\qed
\end{proof}

As a consequence of this theorem, any ordinary binary curve (with
$j = 1/b \ne 0$) can be put in twisted $\bbmu_4$-normal form, via
the map on affine points:
$$
(x,y) \longmapsto (x^2 : x^2 + y : 1 : x^2 + y + x).
$$
In particular all algorithms of this work (over binary fields)
are applicable to the binary NIST curves, which permits
backward compatibility and improved performance.

\begin{commentcode}
\vspace{2mm}
\noindent{\bf Code verification.}
The Weierstrass models can be verified symbolically, in Magma~\cite{Magma}
using the Echidna~\cite{Kohel-Echidna} code, as in
Table~\ref{weierstrass-isom-charp} when $2$ is invertible and
Table~\ref{weierstrass-isom-char2} in characteristic $2$.

\begin{table} 
\caption{Isomorphism with Weierstrass model in characteristic 2}
\label{weierstrass-isom-char2}
\begin{sageblock}
magma.eval("""
FF<b,a> := FunctionField(FiniteField(2),2);
PP<X0,X1,X2,X3> := ProjectiveSpace(FF,3);
CT := EllipticCurve_Twisted_Mu4_NormalForm(PP,b,a);
WT := EllipticCurve([1,a,0,0,b]);
P2<X,Y,Z> := AmbientSpace(WT);
F,G,H := Explode([Y^2 + a*Y*Z + b*Z^2,X^2 + Y*Z,X + a*Z]);
BB_WT_CT := [ [X^2, G, Z^2, G + X*Z],[F + Y*H, F, Z*H, F + X*H] ];
iso_CT_WT := map< CT->WT | [X1+X3,X0+X1,X2], BB_WT_CT >;
""")
\end{sageblock}
\end{table}
\end{commentcode}

\section{Addition algorithms}
\label{sec:addition-algorithms}

We now consider the addition laws for twisted split $\bbmu_4$-normal form.
In the application to prime finite fields of odd characteristic~$p$
(see below for considerations in characteristic 2), under the GRH,
Lagarias, Montgomery and Odlyzko~\cite{LMO-nonresidue} prove a generalization
of the result of Ankeny~\cite{Ankeny}, under which we can conclude
that the least quadratic nonresidue $D \equiv 1 \bmod 4$ is in $O(\log^2(p))$,
and the average value of $D$ is $O(1)$.  Consequently, for a curve
over a finite prime field, one can find small twisting parameters
for constructing the quadratic twist. With this in mind, we ignore
all multiplications by constants $a$ and $D = 4a + 1$.

\begin{theorem}
\label{thm:P1xP1-isomorphism}
Let $C^t$ be an elliptic curve in twisted split $\bbmu_4$-normal form:
$$
X_0^2 - D X_2^2 = c^2 (X_1 X_3 - a (X_1 - X_3)^2),\ X_1^2 - X_3^2 = c^2 X_0 X_2.
$$
over a ring in which $2$ is a unit.
The projections $\pi_1: C^t \rightarrow \PP^1$, with coordinates $(X,Z)$,
given by
$$
\pi_1((X_0:X_1:X_2:X_3)) =
\big\{ ( c X_0  : X_1 + X_3), ( X_1 - X_3 : c X_2 ) \big\},
$$
and $\pi_2: C^t \rightarrow \PP^1$, with coordinates $(Y,W)$,
given by
$$
\pi_2((X_0:X_1:X_2:X_3)) =
\big\{ ( c X_0 : X_1 - X_3 ), ( X_1 + X_3 : c X_2 ) \big\},
$$
determine an isomorphism $\pi_1 \times \pi_2$ with its image:
$$
((c^2/2)^2 X^2 - Z^2)W^2 = D((c^2/2)^2 Z^2 - X^2) Y^2
$$
in $\PP^1 \times \PP^1$, with inverse
$$
\sigma((X:Z),(Y:W)) = ( 2XY: c(XW + ZY): 2ZW : c(ZY - XW) ).
$$
\end{theorem}

\begin{proof}
The morphisms $\sigma$ and $\pi_1 \times \pi_2$ determine
isomorphisms of $\PP^1 \times \PP^1$ with the surface
$X_1^2 - X_3^2 = c^2 X_0 X_2$ in $\PP^3$, and substitution in
the first equation for $C^t$ yields the above hypersurface
in $\PP^1 \times \PP^1$.
\qed
\end{proof}

The twisted split $\bbmu_4$-normal form has $2$-torsion subgroup
generated by $Q = (-c:1:0:1)$ and $R = (0:-1:c:1)$, with $Q + R
= (0:-1:-c:1)$. Over any extension containing a square root
$\varepsilon$ of $-D$, the point $S = (c: -\varepsilon: 0:\varepsilon)$
is a point of order $4$ such that $2S = Q$.

\begin{theorem}
\label{thm:twisted-mu4-addition-law-projections}
Let $C^t$ be an elliptic curve in twisted split $\bbmu_4$-normal form
over a ring in which $2$ is a unit.
The projections $\pi_1$ and $\pi_2$ determine two-dimensional spaces
of bilinear addition law projections:
$$
\begin{array}{r@{\,}c@{\,}l}
\pi_1 \circ \mu(x,y) & = &
\left\{
\begin{array}{l}
\fs_0 = ( U_{13} - U_{31} : U_{20} - U_{02} ),\\
\fs_2 = ( U_{00} + DU_{22} : U_{11} + U_{33} + 2aV_{13} ),
\end{array}
\right. \\[4mm]
\pi_2 \circ \mu(x,y) & = &
\left\{
\begin{array}{l}
\fs_1 = ( U_{13} + U_{31} - 2aV_{13} : U_{02} + U_{20} ),\\
\fs_3 = ( U_{00} - D U_{22} : U_{11} -  U_{33} ),
\end{array}
\right.
\end{array}
$$
where $U_{k\ell} = X_k Y_\ell$ and $V_{k\ell} = (X_k - X_\ell)(Y_k - Y_\ell)$.
The exceptional divisors of the $\fs_j$ are of the form $\Delta_{T_j}
+ \Delta_{T_j+Q}$, where $T_0 = O,\; T_1 = S + R,\; T_2 = R,\; T_3 = S$.
\end{theorem}

\begin{proof}
The existence and dimensions of the spaces of bilinear addition law
projections, as well as the form of the exceptional divisors, follows
from Theorem~26 and Corollary~27 of Kohel~\cite{Kohel}, observing
for $j$ in $\{0, 2\}$ that $T_j + (T_j + Q) = Q$ and for $j$ in $\{1, 3\}$
that $T_j + (T_j + Q) = O$.  The correctness of the forms can be verified
symbolically, 
and the pairs $\{T_j,T_j+Q\}$ determined by the substitution
$(Y_0,Y_1,Y_2,Y_3) = (c,1,0,1)$, as in Corollary~11 of
Kohel~\cite{Kohel}. In particular, for $\fs_0$, we obtain the tuple
$
( U_{13} - U_{31}, U_{20} - U_{02} ) = ( X_1 - X_3, c X_2 ),
$
which vanishes on $\{O,Q\} = \{ (c:1:0:1), (-c:1:0:1) \}$, hence
the exceptional divisor is $\Delta_O + \Delta_Q$.
\qed
\end{proof}

Composing the addition law projections of Theorem~\ref{thm:twisted-mu4-addition-law-projections}
with the isomorphism of Theorem~\ref{thm:P1xP1-isomorphism}, and dividing
by 2, we obtain for the pair $(\fs_0,\fs_1)$ the tuple $(Z_0,Z_1,Z_2,Z_3)$
with
$$
\begin{array}{r@{\,}c@{\,}lr@{\,}c@{\,}l}
      Z_0 & = & (U_{13} - U_{31})(U_{13} + U_{31} - 2aV_{13}), &
Z_1 + Z_3 & = & -c(U_{02} - U_{20})(U_{13} + U_{31} + 2aV_{13}),\\
      Z_2 & = & -(U_{02} - U_{20})(U_{02} + U_{20}), &
Z_1 - Z_3 & = & -c(U_{13} - U_{31})(U_{02} + U_{20}),\\
\end{array}
$$
and for the pair $(\fs_2,\fs_3)$ the tuple $(Z_0,Z_1,Z_2,Z_3)$ with
$$
\begin{array}{r@{\,}c@{\,}lr@{\,}c@{\,}l}
      Z_0 & = & (U_{00} + DU_{22})(U_{00} - DU_{22}), &
Z_1 + Z_3 & = & c\,(U_{11} + U_{33} + 2aV_{13})(U_{00} - DU_{22}),\\
      Z_2 & = & (U_{11} + U_{33} + 2aV_{13})(U_{11} - U_{33}), &
Z_1 - Z_3 & = & c(U_{00} + DU_{22})(U_{11} - U_{33}).
\end{array}
$$
The former have efficient evaluations over a ring in which $2$
is a unit, yielding $(2Z_0, 2Z_1, 2Z_2, 2Z_3)$, and otherwise
we deduce expressions for $(Z_1,Z_3)$:
$$
\begin{array}{l}
Z_1 = c((U_{02} U_{13} - U_{02} U_{31}) - a(U_{02} - U_{20})W_{13}),\\
Z_3 = c((U_{02} U_{31} - U_{20} U_{13}) - a(U_{02} - U_{20})W_{13}),
\end{array}
$$
with $W_{13} = 2(U_{13} + U_{31}) - V_{13}$, and
$$
\begin{array}{l}
Z_1 = c(U_{00} U_{11} - D U_{22} U_{33}) - a(U_{00} - D U_{22})W_{13}),\\
Z_3 = c(U_{00} U_{33} - D U_{22} U_{11}) - a(U_{00} - D U_{22})W_{13}),
\end{array}
$$
with $W_{13} = 2(U_{11} + U_{33}) - V_{13}$, respectively.  We note
that these expressions remain valid over any ring despite the fact
that they were derived via the factorization through the curve in
$\PP^1 \times \PP^1$ which is singular in characteristic~$2$.
\vspace{2mm}

\begin{commentcode}
\noindent{\bf Code verification.} Table~\ref{code:addition-laws} provides
a verification of the correctness of the addition laws.
\vspace{2mm}

\begin{table} 
\caption{Addition laws for split $\bbmu_4$-normal form.}
\label{code:addition-laws}
\begin{sageblock}
magma.eval("""
FF<c,a> := FunctionField(ZZ,2); D := 4*a + 1;
CC := EllipticCurve_Twisted_Split_Mu4_NormalForm(c,a);
mu := AdditionMorphism(CC);
BB := AllDefiningPolynomials(mu)[1];
PPxPP<X0,X1,X2,X3,Y0,Y1,Y2,Y3> := AmbientSpace(Domain(mu));
U00, U11, U22, U33 := Explode([X0*Y0, X1*Y1, X2*Y2, X3*Y3]); // 4M
W13 := (X1 - X3)*(Y1 - Y3); // 1M
A0, A1 := Explode([U00 + D*U22, U00 - D*U22]);
B0, B1 := Explode([U11 - U33, U11 + U33 + 2*a*W13]);
Z0, Z2 := Explode([A0*A1, B0*B1]); // 2M
W1, W3 := Explode([A1*B1, A0*B0]); // 2M
Z1 := c*((U00*U11 - D*U22*U33) + a*A1*W13);
Z3 := c*((U00*U33 - D*U22*U11) + a*A1*W13);
assert BB eq [ 2*Z0, c*(W1 + W3), 2*Z2, c*(W1 - W3) ];
assert BB eq [ 2*Z0, 2*Z1, 2*Z2, 2*Z3 ];
""")
\end{sageblock}
\end{table}
\end{commentcode}

Before evaluating their complexity, we explain the obvious symmetry
of the above equations.  Let $\uptau$ be the translation-by-$R$
automorphism of $C^t$ sending $(X_0:X_1:X_2:X_3)$
to
$$
(X_2:-X_3-2a(X_1+X_3):-DX_0:X_1 + 2a(X_1+X_3)),
$$
and denote also $\uptau$ for the induced automorphism
$$
\uptau((X:Z),(Y:W)) = ((Z:X),(-W:DY))
$$
of its image in $\PP^1 \times \PP^1$. Then for each $(i,j)$ in $(\ZZ/2\ZZ)^2$,
the tuple of morphisms $(\uptau^i \times \uptau^j,\uptau^k)$ such that
$k = i+j$ acts on the set of tuples $(\fs,\fs')$ of addition law projections:
$$
(\uptau^i \times \uptau^j,\uptau^k) \cdot (\fs,\fs') =
\uptau^k \circ ( \fs \circ (\uptau^i \times \uptau^j),\ \fs' \circ (\uptau^i \times \uptau^j) ).
$$

\begin{lemma}
Let $C^t$ be an elliptic curve in split $\bbmu_4$-normal form.
The tuples of addition law projections $(\fs_0,\fs_1)$ and
$(\fs_2,\fs_3)$ are eigenvectors for the action of
$(\uptau \times \uptau, 1)$ and are exchanged, up to scalars,
by the action of $(\uptau \times 1, \uptau)$ and
$(1 \times \uptau, \uptau)$.
\end{lemma}

\begin{proof}
Since an addition law (projection) is uniquely determined by
its exceptional divisor, up to scalars, the lemma follows from
the action of $(\uptau^i \times \uptau^j, \uptau^k)$ on the
exceptional divisors given by Lemma~31 of Kohel~\cite{Kohel},
and can be established directly by substitution.
\qed
\end{proof}

\begin{corollary}
Let $C^t$ be an elliptic curve in twisted split $\bbmu_4$-normal form.
There exists an algorithm for addition with complexity
$11\Mul + 2\mul$ over any ring, and an algorithm with complexity
$9\Mul + 2\mul$ over a ring in which $2$ is a unit.
\end{corollary}

\begin{proof}
Considering the product determined by the pair $(\fs_2,\fs_3)$, the
evaluation of the expressions
$$
\begin{array}{r@{\,}c@{\,}l}
Z_0 & = & (U_{00} - D U_{22})(U_{00} + D U_{22}),\\
Z_2 & = & (U_{11} - U_{33})(U_{11} + U_{33} + 2aV_{13}),
\end{array}
$$
requires $4\Mul$ for the $U_{ii}$ plus $1\Mul$ for $V_{13}$ if $a \ne 0$,
then $2\Mul$ for the evaluation of $Z_0$ and $Z_2$.  Setting $W_{13} =
2(U_{11} + U_{33}) - V_{13}$, a direct evaluation of the expressions
$$
\begin{array}{l}
Z_1 = c((U_{00} U_{11} - D U_{22} U_{33}) - a(U_{00} - D U_{22})W_{13}),\\
Z_3 = c((U_{00} U_{33} - D U_{22} U_{11}) - a(U_{00} - D U_{22})W_{13}),
\end{array}
$$
requires an additional $4\Mul + 2\mul$, saving $1\Mul$ with the relation
$$
(U_{00} - D U_{22}) (U_{11} + U_{33}) =
(U_{00} U_{11} - D U_{22} U_{33}) + (U_{00} U_{33} - D U_{22} U_{11}),
$$
for a complexity of $11\Mul + 2\mul$. If $2$ is a unit, we may instead compute
$$
\begin{array}{r@{\,}c@{\,}l}
Z_1 + Z_3 & = & c\,(U_{00} - D U_{22})(U_{11} + U_{33} + 2aV_{13}),\\
Z_1 - Z_3 & = & c\,(U_{00} + D U_{22})(U_{11} - U_{33}).
\end{array}
$$
and return $(2Z_0,2Z_1,2Z_2,2Z_3)$ using $2\Mul + 2\mul$, for a total
cost of $9\Mul + 2\mul$.
\qed
\end{proof}

\begin{corollary}
Let $C^t$ be an elliptic curve in twisted split $\bbmu_4$-normal form.
There exists an algorithm for doubling with complexity
$6\Mul + 5\Sqr + 2\mul$ over any ring, and an algorithm with complexity
$4\Mul + 5\Sqr + 2\mul$ over a ring in which $2$ is a unit.
\end{corollary}

\begin{proof}
The specialization to $X_i = Y_i$ gives:
$$
\begin{array}{r@{\,}c@{\,}l}
Z_0 & = & (X_0^2 - D X_2^2)(X_0^2 + D X_2^2),\\
Z_2 & = & (X_1^2 - X_3^2)(X_1^2 + X_3^2 + 2a(X_1 + X_3)^2).
\end{array}
$$
The evaluation of $X_i^2$ costs $4\Sqr$ plus $1\Sqr$ for $(X_1+X_3)^2$
if $a \ne 0$, rather than $4\Mul$ + $1\Mul$.
Setting $W_{13} = 2(X_1^2 + X_3^2) - (X_1 + X_3)^2\, [= (X_1 - X_3)^2]$,
a direct evaluation of the expressions
$$
\begin{array}{l}
Z_1 = c((X_0^2 X_1^2 - D X_2^2 X_3^2) - a(X_0^2 - D X_2^2)W_{13}),\\
Z_3 = c((X_0^2 X_3^2 - D X_2^2 X_1^2) - a(X_0^2 - D X_2^2)W_{13}),
\end{array}
$$
requires an additional $4\Mul + 2\mul$, as above, for a complexity
of $6\Mul + 5\Sqr + 2\mul$.  If $2$ is a unit, we compute
$$
\begin{array}{r@{\,}c@{\,}l}
Z_1 + Z_3 & = & c\,(X_0^2 - D X_2^2)(X_1^2 + X_3^2 + 2a(X_1 + X_3)^2),\\
Z_1 - Z_3 & = & c\,(X_0^2 + D X_2^2)(X_1^2 - X_3^2).
\end{array}
$$
using $2\Mul + 2\mul$, which gives $4\Mul + 5\Sqr + 2\mul$.
\qed
\end{proof}

In the next section we explore efficient algorithms for evaluation
of the addition laws and doubling forms in characteristic~$2$.

\section{Binary addition algorithms}
\label{sec:binary-addition-algorithms}

Suppose that $k$ is a finite field of characteristic~2. The Artin-Schreier
extension $k[\wx]/k$ over which we twist is determined by the additive
properties of $a$, and half of all elements of $k$ determine the same field
(up to isomorphism) and hence an isomorphic twist. For instance, if $k/\FF_2$
is an odd degree extension, we may take $a = 1$. As above, we assume that
that multiplication by $a$ is neglible in our complexity analyses.

\begin{theorem}
\label{thm:binary-twisted-mu4-addition-laws}
Let $C^t$ be an elliptic curve in twisted split $\bbmu_4$-normal form:
$$
X_0^2 + X_2^2 = c^2 (X_1 X_3 + a (X_1 + X_3)^2),\ X_1^2 + X_3^2 = c^2 X_0 X_2,
$$
over a field of characteristic $2$.  A complete system of addition
laws is given by the two maps $\fs_0$ and $\fs_2$,
$$
\begin{array}{c}
\big( 
    (U_{13} + U_{31})^2, 
    c (U_{02} U_{31} + U_{20} U_{13} + aF), 
    (U_{02} + U_{20})^2, 
    c (U_{02} U_{13} + U_{20} U_{31} + aF) 
  \big),\\[1mm]
\big( 
    (U_{00} + U_{22})^2, 
    c (U_{00} U_{11} + U_{22} U_{33} + aG), 
    (U_{11} + U_{33})^2, 
    c (U_{00} U_{33} + U_{11} U_{22} + aG)  
  \big),
\end{array}
$$
respectively, where $U_{jk} = X_j Y_k$ and
$$
F = (X_1 + X_3) (Y_1 + Y_3) (U_{02} + U_{20}) \mbox{ and }
G = (X_1 + X_3) (Y_1 + Y_3) (U_{00} + U_{22}).
$$
The respective exceptional divisors are $4\Delta_O$ and $4\Delta_{S}$ where
$S = (1:c:1:0)$ is a $2$-torsion point.
\end{theorem}

\begin{proof}
The addition laws $\fs_0$ and $\fs_2$ are the conjugate addition laws
of Theorem~\ref{thm:mu4-addition-laws} (as can be verified symbolically)\footnote{
As is verified by the implementation in Echidna~\cite{Kohel-Echidna}
written in Magma~\cite{Magma}.}
and, equivalently, are described by the reduction at $2$ of the addition
laws derived from the tuples of addition law projections $(\fs_0,\fs_1)$
and $(\fs_2,\fs_3)$ of Theorem~\ref{thm:twisted-mu4-addition-law-projections}.
Since the points $O$ and $S$ are fixed rational points of the twisting
morphism, the exceptional divisors are of the same form.
As the exceptional divisors are disjoint, the pair of addition laws form
a complete set.
\qed
\end{proof}

\noindent{\bf Remark.}
Recall that the addition laws $\fs_1$ and $\fs_3$ on the split
$\bbmu_4$-normal form have exceptional divisors $4\Delta_T$
and $4\Delta_{-T}$ in characteristic $2$ (since $S = O$).
Consequently their conjugation by the twisting
morphism yields a conjugate pair over $k[\wx]$, since the twisted
curve does not admit a $k$-rational $4$-torsion point $T$.
There exist linear combinations of these twisted addition laws
which extend the set $\{\fs_0,\fs_2\}$ to a basis over $k$
(of the space of dimension four), but they do not have such
an elegant form as $\fs_0$ and $\fs_2$.

\begin{corollary}
\label{cor:binary-twisted-split-mu4-complexity}
Let $C^t$ be an elliptic curve in twisted split $\bbmu_4$-normal form over
a field of characteristic 2. There exists an algorithm for addition with
complexity $9\Mul + 2\Sqr + 2\mul$.
\end{corollary}

\begin{proof}
Since the addition laws differ from the split $\bbmu_4$-normal form only by
the term $aF$ (or $aG$), it suffices to determine the complexity of its
evaluation.  Having determined $(U_{02},U_{20})$ (or $(U_{00},U_{22})$), we
require an additional $2\Mul$, which gives the complexity bound.
\qed
\end{proof}

For the $\bbmu_4$-normal form the addition law, after coefficient scaling, we
find that the addition law with exceptional divisor $4\Delta_O$ takes the form
$$
(
  (U_{13} + U_{31})^2, U_{02} U_{31} + U_{20} U_{13} + aF,
  (U_{20} + U_{02})^2, U_{02} U_{13} + U_{20} U_{31} + aG
),
$$
and in particular does not involve multiplication by constants (other than $a$
which we may take in $\{0,1\}$ in cryptographic applications). This gives the
following complexity result.

\begin{corollary}
\label{cor:binary-twisted-mu4-complexity}
Let $C^t$ be an elliptic curve in twisted $\bbmu_4$-normal form over a field
of characteristic 2. There exists an algorithm for addition outside of the
diagonal $\Delta_O$ with complexity $9\Mul + 2\Sqr$.
\end{corollary}

\section{Binary doubling algorithms}
\label{sec:binary-doubling}

We recall the hypothesis that multiplication by $a$ is negligible.
In the cryptographic context (e.g.~in application to the binary NIST
curves), we may assume $a = 1$ (or $a = 0$ for the untwisted forms).

\begin{corollary}
Let $C^t$ be an elliptic curve in twisted split $\bbmu_4$-normal form.
The doubling map is uniquely determined by
$$
\begin{array}{r@{\,}l}
(&(X_0 + X_2)^4 :
 c((X_0 X_3 + X_1 X_2)^2 + a(X_0 + X_2)^2(X_1 + X_3)^2) : \\
 &(X_1 + X_3)^4 :
 c((X_0 X_1 + X_2 X_3)^2 + a(X_0 + X_2)^2(X_1 + X_3)^2)\,)
\end{array}
$$
\end{corollary}

\begin{proof}
This follows from specializing $X_j = Y_j$ in the form $\fs_2$ of
Theorem~\ref{thm:binary-twisted-mu4-addition-laws}.
\qed
\end{proof}

We note that in cryptographic applications we may assume that $a = 0$ (untwisted form),
giving
$$
((X_0 + X_2)^4 : c(X_0 X_3 + X_2 X_1)^2 : (X_1 + X_3)^4 : c(X_0 X_1 + X_2 X_3)^2),
$$
and otherwise $a = 1$, in which case we have
$$
((X_0 + X_2)^4 : c(X_0 X_1 + X_2 X_3)^2 : (X_1 + X_3)^4 : c(X_0 X_3 + X_2 X_1)^2).
$$
It is clear that the evaluation of doubling on the twisted and untwisted
normal forms is identical.  This is true also for the case of general $a$,
up to the computation of $(X_0 + X_2)^2(X_1 + X_3)^2$.
We nevertheless give an algorithm which improves upon the number of constant
multiplications reported in Kohel~\cite{Kohel-Indocrypt}, in terms of
polynomials in $u = c^{-1}$. With this notation, we note that the defining
equations of the curve are:
$$
\begin{array}{l}
X_1 X_3 = u^2(X_0 + X_2)^2,\\
X_0 X_2 = u^2(X_1 + X_3)^2.
\end{array}
$$
These relations are important, since they permit us to replace any instances
of the multiplications on the left with the squarings on the right.
As a consequence, we have
$$
\begin{array}{r@{\;}l}
X_0 X_1 + X_2 X_3 & = (X_0 + X_3)(X_2 + X_1) + X_0 X_2 + X_1 X_3\\
                  & = (X_0 + X_3)(X_2 + X_1) + u^2((X_0 + X_2)^2 + (X_1 + X_3)^2)\\
X_0 X_3 + X_2 X_1 & = (X_0 + X_1)(X_2 + X_3) + X_0 X_2 + X_1 X_3\\
                  & = (X_0 + X_1)(X_2 + X_3) + u^2((X_0 + X_2)^2 + (X_1 + X_3)^2).
\end{array}
$$
Moreover these forms are linearly dependent with $(X_0 + X_2)(X_1 + X_3)$
$$
(X_0 X_1 + X_2 X_3) + (X_0 X_3 + X_2 X_1) = (X_0 + X_2)(X_1 + X_3),
$$
so that two multiplications are sufficient for the determination of these
three forms.  Putting this together, it suffices to evaluate the tuple
$$
(u(X_0 + X_2)^4, (X_0 X_1 + X_2 X_3)^2, u(X_1 + X_3)^4, (X_0 X_3 + X_2 X_1)^2),
$$
for which we obtain the following complexity for doubling.

\begin{corollary}
Let $C^t$ be a curve in twisted split $\bbmu_4$-normal form.
There exists an algorithm for doubling with complexity
$2\Mul + 5\Sqr + 3\mul_u$.
\end{corollary}
Using the semisplit $\bbmu_4$-normal form, the complexity of $2\Mul + 5\Sqr + 2\mul_u$
of Kohel~\cite{Kohel-Indocrypt}, saving one constant multiplication, carries over to
the corresponding twisted semisplit $\bbmu_4$-normal form (referred to as nonsplit).
By a similar argument the same complexity, $2\Mul + 5\Sqr + 2\mul_u$, is obtained for
the $\bbmu_4$-normal form of this article.

\section{Montgomery endomorphisms of Kummer products}
\label{sec:montgomery-endomorphism}

We recall certain results of Kohel~\cite{Kohel-Indocrypt} concerning
the Montgomery endomorphism with application to scalar multiplication
on products of Kummer curves. We define the Montgomery endomorphism
to be the map $\varphi:C \times C \rightarrow C \times C$
given by $(Q,R) \mapsto (2Q,Q+R)$.  With a view to scalar multiplication,
this induces
$$
((n+1)P,nP) \longmapsto ((2n+2)P,(2n+1)P),
$$
and
$$
(nP,(n+1)P) \longmapsto (2nP,(2n+1)P).
$$
By exchanging the order of the coordinates on input and output, an algorithm for
the Montgomery endomorphism computes $((2n+2)P,(2n+1)P)$ or $((2n+1)P,2nP)$ from
the input point $((n+1)P,nP)$. This allows us to construct a symmetric algorithm
for the scalar multiple $kP$ of $P$ via a Montgomery ladder
$$
((n_{i}+1)P,n_{i}P) \longmapsto ((n_{i+1}+1)P,n_{i+1}P) =
\left\{\begin{array}{@{}l}
((2n_{i}+1)P,2n_{i}P),\mbox{ or}\\
((2n_{i}+2)P,(2n_{i}+1)P).
\end{array}
\right.
$$
It is noted that the Montgomery endomorphism sends each of the curves
$$
\Delta_P = \{ (Q,Q-P) \;|\; Q \in C(\bar{k}) \},
\mbox{ and }
\Delta_{-P} = \{ (Q,Q-P) \;|\; Q \in C(\bar{k}) \},
$$
to itself, and exchange of coordinates induces $\Delta_P \rightarrow \Delta_{-P}$.

We now assume that $C$ is a curve in split $\bbmu_4$-normal form,
and define the Kummer curve $\cK(C) = C/\{\pm 1\} \isom \PP^1$,
equipped with map
$$
\pi((X_0:X_1:X_2:X_3) = \left\{
\begin{array}{@{}c}
(cX_0 : X_1 + X_3),\\
(X_1 - X_3 : cX_2).
\end{array}
\right.
$$
This determines a curve $\cK(\Delta_P)$ as the image of $\Delta_P$
in $\cK(C) \times \cK(C)$.

\begin{lemma}
For any point $P$ of $C$, the Montgomery-oriented curve $\cK(\Delta_P)$ equals $\cK(\Delta_{-P})$.
\end{lemma}

\begin{proof}
It suffices to note that $(\overline{Q},\overline{Q-P}) \in \cK(\Delta_{P})(\bar{k})$
is also a point of $\cK(\Delta_{-P})$:
$$
(\overline{Q},\overline{Q-P})
  = (\overline{-Q},\overline{-Q+P})
  = (\overline{-Q},\overline{-Q-(-P)}) \in \cK(\Delta_{-P}),
$$
hence $\cK(\Delta_P) \subseteq \cK(\Delta_{-P})$ and by symmetry $\cK(\Delta_{-P})
\subseteq \cK(\Delta_{P})$.
\end{proof}

We conclude, moreover, that $\cK(\Delta_P)$ is well-defined by a point on
the Kummer curve.
\begin{lemma}
The Montgomery-oriented curve $\cK(\Delta_P)$ depends only on $\pi(P)$.
\end{lemma}

\begin{proof}
The dependence only on $\pi(P)$ is a consequence of the previous lemmas,
which we make explicit here.  Let $P = (s_0:s_1:s_2:s_3)$ and
$\pi(P) = (t_0:t_1)$.  By Theorem~24 of Kohel~\cite{Kohel-Indocrypt},
the curve $\cK(\Delta_P)$ takes the form,
$$
s_0(U_0V_1 + U_1V_0)^2 + s_2(U_0V_0 + U_1V_1)^2 = c(s_1 + s_3)U_0U_1V_0V_1,
$$
but then $(s_0:s_1+s_3:s_2) = (t_0^2:c\,t_0t_1,t_1^2)$ in $\PP^2$, hence
$$
t_0^2(U_0V_1 + U_1V_0)^2 + t_1^2(U_0V_0 + U_1V_1)^2 = c^2t_0t_1U_0U_1V_0V_1.
$$
which shows that the curve depends only on $\pi(P)$.
\end{proof}

We note similarly that the Kummer curve $\cK(C) = \cK(C^t)$ is independent
of the quadratic twist, in the sense that any twisting isomorphism
$\tau : C \rightarrow C^t$ over $\bar{k}$ induces a unique isomorphism
$\cK(C) \rightarrow \cK(C^t)$.  One can verify directly the twisting
isomorphism $\tau$ of Theorem~\ref{thm:twisted-split-mu4-curve} induces
the identity on the Kummer curves with their given projections.
We thus identify $\cK(C) = \cK(C^t)$, and denote $\pi: C \rightarrow \cK(C)$
and $\pi^t:C^t \rightarrow \cK(C)$ the respective covers of the Kummer curve.

\begin{theorem}
Let $C$ be a curve in split $\bbmu_4$-normal form and $C^t$ be a quadratic twist
over the field~$k$.  If $P \in C^t(\bar{k})$ and $Q \in C(\bar{k})$ such that
$\pi^t(P) = \pi(Q)$, then $\cK(\Delta_P) = \cK(\Delta_{Q})$.
\end{theorem}

It follows that we can evaluate the Montgomery endomorphism on $\cK(\Delta_P)$,
for $P \in C^t(k)$, and $\pi(P) = (t_0:t_1)$, using the same algorithm and with
the same complexity as in Kohel~\cite{Kohel-Indocrypt}.
We recall the complexity result here, assuming a normalisation $t_0 = 1$ or
$t_1 = 1$.

\begin{corollary}
The Montgomery endomorphism on $\cK(\Delta_P)$ can be computed with
$4\Mul + 5\Sqr + 1\mul_t + 1\mul_c$ or with $4\Mul + 4\Sqr + 1\mul_t + 2\mul_c$.
\end{corollary}

By the same argument, the same Theorem~24 of Kohel~\cite{Kohel-Indocrypt}
provides the necessary map for point recovery in terms of the input point
$P = (s_0:s_1:s_2:s_3)$ of $C^t(k)$.

\begin{theorem}
\label{thm:mu4-Montgomery-oriented}
Let $C^t$ be an elliptic curve in twisted split $\bbmu_4$-normal
form with rational point $P = (s_0:s_1:s_2:s_3)$.
If $P$ is not a $2$-torsion point, the morphism $\lambda: C \rightarrow \cK(\Delta_P)$
is an isomorphism, and defined by
$$
\begin{array}{r@{\,}c@{\,}l}
\pi_1 \circ \lambda (X_0:X_1:X_2:X_3) & = &
\left\{
\begin{array}{l}
(c X_0 : X_1 + X_3),\\
(X_1 + X_3 : c X_2),
\end{array}
\right.\\[4mm]
\pi_2 \circ \lambda (X_0:X_1:X_2:X_3) & = &
\left\{
\begin{array}{l}
(s_0 X_0 + s_2 X_2 : s_1 X_1 + s_3 X_3),\\
(s_3 X_1 + s_1 X_3 : s_2 X_0 + s_0 X_2),
\end{array}
\right.
\end{array}
$$
with inverse $\lambda^{-1}((U_0:U_1),(V_0:V_1))$ equal to
\ignore{
$$
\left\{
\begin{array}{@{}l}
(
  (s_1 + s_3) U_0^2 V_0 :
  (s_0 U_0^2 + s_2 U_1^2) V_1 + c s_1 U_0 U_1 V_0 :
  (s_1 + s_3) U_1^2 V_0 :
  (s_0 U_0^2 + s_2 U_1^2) V_1 + c s_3 U_0 U_1 V_0
),\\
(
  (s_1 + s_3) U_0^2 V_1 :
  (s_2 U_0^2 + s_0 U_1^2) V_0 + c s_3 U_0 U_1 V_1 :
  (s_1 + s_3) U_1^2 V_1 :
  (s_2 U_0^2 + s_0 U_1^2) V_0 + c s_1 U_0 U_1 V_1
).
\end{array}
\right.
$$
}

\vspace{2mm}
\noindent\hspace{-8mm}\scalebox{0.90}{
$
\left\{
\begin{array}{@{}l}
(
  (s_1 + s_3) U_0^2 V_0 :
  (s_0 U_0^2 + s_2 U_1^2) V_1 + c s_1 U_0 U_1 V_0 :
  (s_1 + s_3) U_1^2 V_0 :
  (s_0 U_0^2 + s_2 U_1^2) V_1 + c s_3 U_0 U_1 V_0
),\\
(
  (s_1 + s_3) U_0^2 V_1 :
  (s_2 U_0^2 + s_0 U_1^2) V_0 + c s_3 U_0 U_1 V_1 :
  (s_1 + s_3) U_1^2 V_1 :
  (s_2 U_0^2 + s_0 U_1^2) V_0 + c s_1 U_0 U_1 V_1
).
\end{array}
\right.
$
} 
\end{theorem}

This allows for the application of the Montgomery endomorphism to scalar multiplication
on $C^t$. Using the best results of the present work, the complexity is comparable to a
double and add algorithm with window of width $4$.

\begin{commentcode}
\begin{table} 
\caption{Montgomery endomorphism point recovery}
\label{code:montgomery-recovery}
\begin{sageblock}
magma.eval("""
FF<c,a> := FunctionField(GF(2),2);
PP<X0,X1,X2,X3> := ProjectiveSpace(FF,3);
CC := EllipticCurve_Twisted_Split_Mu4_NormalForm(PP,c,a);
CCxCC := Domain(AdditionMorphism(CC));
PPxPP<X0,X1,X2,X3,S0,S1,S2,S3> := Ambient(CCxCC);
SS := [S0,S1,S2,S3];
P1<U0,U1> := Curve(ProjectiveSpace(FF,1));
BB_pi_1 := [[c*X0,X1+X3],[X1-X3,c*X2]];
BB_mu_1 := [[X0*S0 + X2*S2, X1*S1 + X3*S3], [X3*S1 + X1*S3, X2*S0 + X0*S2]];
pi_1 := map< CCxCC->P1 | BB_pi_1 >;
mu_1 := map< CCxCC->P1 | BB_mu_1 >;
BB_12 := [S1 cat S2 cat SS : S1 in BB_pi_1, S2 in BB_mu_1 ];
K1xK1xP3<U0,U1,V0,V1,S0,S1,S2,S3> := ProductProjectiveSpace(FF,[1,1,3]);
SS := [S0,S1,S2,S3];
KCxCC := Scheme(K1xK1xP3,[
    S0*(U0*V1+U1*V0)^2 + S2*(U0*V0 + U1*V1)^2 + c*(S1+S3)*U0*U1*V0*V1,
    (S0+S2)^2 + c^2*(S1*S3 + a*(S1+S3)^2), (S1+S3)^2 + c^2*S0*S2 ]);
BB_phi_1 := [U0^4+U1^4,c^2*U0^2*U1^2];
BB_phi_2 := [
    [(S1+S3)*(U0*V0+U1*V1)^2,c*S0*(U0*V1+U1*V0)^2],
    [c*S2*(U0*V0+U1*V1)^2,(S1+S3)*(U0*V1+U1*V0)^2]];
mont_CC := map< KCxCC->KCxCC | [ BB_phi_1 cat S2 cat SS : S2 in BB_phi_2 ] >;
lambda_CC := map< CCxCC->KCxCC | BB_12 >;
BB_21 := [ [
    (S1 + S3)*U0^2*V0,
    (S0*U0^2 + S2*U1^2)*V1 + c*S1*U0*U1*V0,
    (S1 + S3)*U1^2*V0,
    (S0*U0^2 + S2*U1^2)*V1 + c*S3*U0*U1*V0 ],
    [
    (S1 + S3)*U0^2*V1,
    (S2*U0^2 + S0*U1^2)*V0 + c*S3*U0*U1*V1,
    (S1 + S3)*U1^2*V1,
    (S2*U0^2 + S0*U1^2)*V0 + c*S1*U0*U1*V1
    ] ];
lambda_inv := map< KCxCC->CC | BB_21 >;
""")
\end{sageblock}
\end{table}
\end{commentcode}

\section{Conclusion}
\label{sec:conclusion}

Elliptic curves in the twisted $\bbmu_4$-normal form of this article
(including split and semisplit variants) provide models for curves
which, on the one hand, are isomorphic to twisted Edwards curves with
efficient arithmetic over nonbinary fields, and, on the other,
have good reduction and efficient arithmetic in characteristic~$2$.

Taking the best reported algorithms from the EFD~\cite{EFD},
we conclude with a tabular comparison of the previously best known
complexity results for doubling and addition algorithms on
projective curves.
We include the projective lambda model (a~singular quartic model in $\PP^2$),
which despite the extra cost of doubling, admits a slightly better algorithm
for addition than L\'opez-Dahab (see~\cite{OLAR-Lambda}).
Binary Edwards curves~\cite{BinaryEdwards}, like the twisted $\bbmu_4$-normal
form of this work, cover all ordinary curves, but the best
complexity result we give here is for $d_1 = d_2$ which has a
rational $4$-torsion point (corresponding to the trivial twist,
for which the $\bbmu_4$-normal form gives better performance).
Similarly, the L\'opez-Dahab model with $a_2=0$ admits a rational
$4$-torsion point, hence covers the same classes, but the fastest
arithmetic is achieved on the quadratic twists with $a_2=1$, which
manage to save one squaring $\Sqr$ for doubling relative to the present
work, at the loss of generality (one must vary the weighted
projective space according to the twist, $a_2 = 0$ or $a_2 = 1$)
and with a large penalty for the cost of addition.
The results stated here concern the twisted $\bbmu_4$-normal form which minimize
the constant multiplications.  In the final columns, we indicate the fractions
of ordinary curves covered by the model (assuming a binary field of odd degree),
and whether the family includes the NIST curves.
\vspace{2mm}

\begin{center}
\begin{tabular}{|@{\;}c|@{\;}l|@{\;}l|@{\;}c|@{\;}c|}
\hline
Curve model    & Doubling & Addition & $\%$ & NIST\\
\hline
Lambda coordinates     & $3\Mul + 4\Sqr + 1\mul$     & $11\Mul + 2\Sqr$ & 100\% & \cmark\\
Binary Edwards ($d_1=d_2$) & $2\Mul + 5\Sqr + 2\mul$ & $16\Mul + 1\Sqr + 4\mul$ &\ 50\% & \xmark\\
L\'opez-Dahab ($a_2=0$)  & $2\Mul + 5\Sqr + 1\mul$ & $14\Mul + 3\Sqr$  &\ 50\% & \xmark \\
L\'opez-Dahab ($a_2=1$)  & $2\Mul + 4\Sqr + 2\mul$ & $13\Mul + 3\Sqr$  &\ 50\% & \cmark \\
Twisted $\bbmu_4$-normal form  & $2\Mul + 5\Sqr + 2\mul$ &\;\;$9\Mul + 2\Sqr$ & 100\% & \cmark \\
$\bbmu_4$-normal form  & $2\Mul + 5\Sqr + 2\mul$ &\;\;$7\Mul + 2\Sqr$  &\ 50\% & \xmark \\
\hline
\end{tabular}\\
\vspace{4mm}
Table of binary doubling and addition algorithm complexities.
\end{center}
\vspace{2mm}

All curves can be represented in lambda coordinates or in $\bbmu_4$-normal form.
However by considering the two cases $a_2 \in \{0,1\}$, as for the L\'opez-Dahab
models, the twists of the $\bbmu_4$-normal form with $a_2 = 0$ give the faster
$\bbmu_4$-normal form and only when $a_2 = 1$ does one need the twisted model with
its reduced complexity.

By consideration of twists, we are able to describe a uniform family
of curves which capture nearly optimal known doubling performance of
binary curves (up to $1\Sqr$), while vastly improving the performance
of addition algorithms applicable to all binary curves.
By means of a trivial encoding in twisted $\bbmu_4$-normal form
(see Corollary~\ref{cor:binary-weierstrass}), this brings efficient
arithmetic of these $\bbmu_4$-normal forms to binary NIST curves.


\newpage

\begin{thebibliography}{99}

\bibitem{Ankeny}
N.~C.~Ankeny,
The least quadratic non residue.
{\it Annals of Mathematics, second series}, {\bf 55}, no. 1, 65--72, 1952.

\bibitem{BernsteinLange-Edwards}
D.~J.~Bernstein, T.~Lange,
Faster addition and doubling on elliptic curves.
{\it Advances in cryptology---ASIACRYPT 2007},
{\it Lecture Notes in Computer Science}, {\bf 4833}, 29--50, 2007.

\bibitem{BinaryEdwards}
D.~J.~Bernstein, T.~Lange, R.~Rezaeian Farashahi,
Binary Edwards curves.
{\it Cryptographic hardware and embedded systems (CHES 2008, Washington, D.C.)},
{\it Lecture Notes in Computer Science}, {\bf 5154}, 244--265, 2008.

\bibitem{EFD}
D.~J.~Bernstein, T.~Lange,
Explicit for\-mu\-las data\-base.
\url{http://www.hyperelliptic.org/EFD/}

\bibitem{BernsteinEtAl-TwistedEdwards}
D.~J.~Bernstein, P.~Birkner, M.~Joye, T.~Lange, C.~Peters,
Twisted Edwards curves.
{\it Progress in cryptology -- AFRICACRYPT 2008},
{\it Lecture Notes in Computer Science}, {\bf 5023}, 389--405, 2008.

\bibitem{ChudnovskyBrothers}
D.~V.~Chudnovsky and G.~V.~Chudnovsky.
Sequences of numbers generated by addition in formal groups and new primality and factorization tests.
{\it Adv. in Appl. Math.}, {\bf 7}, (4), 385--434, 1986.

\bibitem{Edwards}
H.~Edwards.
A normal form for elliptic curves.
{\it Bulletin of the American Mathematical Society}, {\bf 44}, 393--422, 2007.

\bibitem{Hisil-EdwardsRevisited}
H.~Hisil, K.~K.-H.~Wong, G.~Carter, E.~Dawson,
Twisted Edwards curves revisited.
{\it Advances in cryptology -- ASIACRYPT 2008},
{\it Lecture Notes in Computer Science}, {\bf 5350}, Springer, 326--343, 2008.



\bibitem{Kohel}
D.~Kohel,
Addition law structure of elliptic curves.
{\it Journal of Number Theory} {\bf 131}, Issue 5, 894--919, 2011.

\bibitem{Kohel-AGCT}
D.~Kohel,
A normal form for elliptic curves in characteristic 2.
talk at {\it Arithmetic, Geometry, Cryptography and Coding Theory},
Luminy, 15 March 2011.
\url{http://iml.univ-mrs.fr/~kohel/pub/normal_form.pdf}.

\bibitem{Kohel-Indocrypt}
D.~Kohel,
Efficient arithmetic of elliptic curves in characteristic 2.
{\it Advances in cryptology---INDOCRYPT 2012 (Kolkata, 2012)},
{\it Lecture Notes in Computer Science},
{\bf 7668}, 378--398, 2012.

\bibitem{Kohel-Echidna}
D.~Kohel et al., Echidna algorithms, v.4.0, 2013.
URL: \url{http://echidna.maths.usyd.edu.au/echidna/index.html}

\bibitem{LMO-nonresidue}
J.~C.~Lagarias, H.~L.~Montgomery, and A.~M.~Odlyzko,
A bound for the least prime ideal in the Chebotarev density theorem,
{\it Invent.{} Math.{}}, {\bf 54}, 271--296, 1979.


\bibitem{LangeRuppert}
H.~Lange and W.~Ruppert.
Complete systems of addition laws on abelian varieties.
{\em Invent.{} Math.}, {\bf 79} (3), 603--610, 1985.

\bibitem{Magma}
Magma Computational Algebra System (Version 2.20),
\url{http://magma.maths.usyd.edu.au/magma/handbook/}, 2015.

\bibitem{OLAR-Lambda}
T.~Oliveira, J.~L\'opez, D.~F.~Aranha, F.~Rodr\'iguez-Henri\'iquez,
Lambda coordiantes for binary elliptic curves,
{\it Cryptographic Hardware and Embedded Systems - CHES 2013}
{\it Lecture Notes in Computer Science}, {\bf 8086}, 311--330, 2013.

\bibitem{Sage}
W.~A.~Stein et al., Sage Mathematics Software (Version 6.0).
The Sage Development Team, 2015, \url{http://www.sagemath.org}.
\end{thebibliography}
\end{document}